\documentclass[oneside]{amsart}
\usepackage[T1]{fontenc}
\usepackage[utf8]{inputenc}
\usepackage[french,english]{babel}

\usepackage{amsmath, amsthm, amsfonts}
\usepackage{amssymb}
\usepackage{latexsym}
\usepackage{verbatim}
\usepackage{hyperref}
\usepackage{color}
\definecolor{mylinkcolor}{rgb}{0.5,0.0,0.0}
\definecolor{myurlcolor}{rgb}{0.0,0.0,0.75}
\hypersetup{colorlinks=true,urlcolor=myurlcolor,citecolor=myurlcolor,linkcolor=mylinkcolor,linktoc=page,breaklinks=true}
\usepackage{textcomp}
\usepackage[all,cmtip]{xy}
\usepackage{tabularx}
\usepackage{enumerate}
\usepackage{amssymb}
\usepackage{array}
\usepackage{colonequals}
\usepackage{booktabs}

\DeclareMathAlphabet{\mathfr}{U}{euf}{m}{n}

\newtheorem{theorem}{Theorem}

\newtheorem{proposition}[theorem]{Proposition}
\newtheorem{corollary}[theorem]{Corollary}

\newtheorem{lemma}[theorem]{Lemma}

\theoremstyle{remark}
\newtheorem{remark}[theorem]{Remark}

\newcommand{\Q}{\mathbb Q}
\newcommand{\Qbar}{{\overline{\mathbb Q}}}

\newcommand{\R}{\mathbb R}
\newcommand{\Z}{\mathbb Z}

\providecommand{\C}{\mathbb C}
\renewcommand{\C}{\mathbb C}
\newcommand{\HH}{\mathbb H}

\newcommand{\GL}{\mathrm{GL}}
\newcommand{\M}{\mathrm{M }}

\newcommand{\End}{\operatorname{End}}

\newcommand{\USp}{\operatorname{USp}}

\newcommand{\Aut}{\operatorname{Aut}}

\newcommand{\GSp}{\operatorname{GSp}}
\newcommand{\Tr}{\operatorname{Tr}}

\newcommand{\ST}{\mathrm{ST}}

\newcommand{\AST}{\mathrm{AST}}
\newcommand{\rk}{\mathrm{rk}}

\newcommand{\NS}{\mathrm{NS}}
\newcommand{\et}{\acute{\mathrm{e}}\mathrm{t}}
\newcommand{\I}{\mathrm{(I)}}
\newcommand{\II}{\mathrm{(II)}}
\newcommand{\III}{\mathrm{(III)}}
\newcommand{\IV}{\mathrm{(IV)}}

\usepackage{url}

\newcommand{\mat}{\mathsf{M}}

\newcommand{\newabstract}[1]{%
  \par\bigskip
  \csname otherlanguage*\endcsname{#1}%
  \csname captions#1\endcsname
  \item[\hskip\labelsep\scshape\abstractname.]
  }

\begin{document}
\title[Arithmetic invariants from Sato--Tate moments]{Arithmetic invariants from Sato--Tate moments}

\author{Edgar Costa}
\address{Department of Mathematics,
Massachusetts Institute of Technology,
77 Massachusetts Ave., Cambridge, MA 02139, United States}
\email{edgarc@mit.edu}
\urladdr{\url{https://edgarcosta.org}}

\author{Francesc Fit\'e}
\address{Department of Mathematics,
Massachusetts Institute of Technology,
77 Massachusetts Ave., Cambridge, MA 02139, United States}
\email{ffite@mit.edu}
\urladdr{\url{https://math.mit.edu/~ffite/}}

\author{Andrew V. Sutherland}
\address{Department of Mathematics,
Massachusetts Institute of Technology,
77 Massachusetts Ave., Cambridge, MA 02139, United States}
\email{drew@math.mit.edu}
\urladdr{\url{https://math.mit.edu/~drew/}}
\date{\today}

\begin{abstract}
  We give some arithmetic-geometric interpretations of the moments $\M_2[a_1]$, $\M_1[a_2]$, and $\M_1[s_2]$ of the Sato--Tate group of an abelian variety $A$ defined over a number field by relating them to the ranks of the endomorphism ring and N\'eron--Severi group of $A$.

\end{abstract}
\maketitle

Let $A$ be an abelian variety of dimension $g\geq 1$ defined over a number field $k$. For a rational prime $\ell$, let
$$
\rho_{A,\ell}\colon G_k\rightarrow \Aut(V_\ell(A))
$$
denote the $\ell$-adic representation attached to $A$ given by the action of the absolute Galois group of $G_k$ on the rational Tate module of $A$. Let $G_\ell$ denote the Zariski closure of the image of $\rho_{\ell,A}$, viewed as a subgroup scheme of $\GSp_{2g}$, let~$G_\ell^1$ denote the kernel of the restriction to $G_\ell$ of the similitude character, and fix an embedding $\iota$ of $\Q_\ell$ into $\C$. The \emph{Sato--Tate group} $\ST(A)$ of $A$ is a maximal compact subgroup of the $\C$-points of the base change $G_\ell^1\times _{\Q_\ell,\iota}\C$ (see \cite[\S2]{FKRS12} and \cite[Chap.~8]{Ser12}).

Throughout this note we shall assume that 
the algebraic Sato--Tate conjecture of Banaszak and Kedlaya~\cite[Conjecture 2.3]{BK16a} holds for $A$. This conjecture is known, for example, when $g\leq 3$ (see \cite[Thm. 6.10]{BK16a}), or more generally, whenever the Mumford--Tate conjecture holds for $A$ (see \cite{CC}). It predicts the existence of an algebraic reductive group $\AST(A)$ defined over $\Q$ such that
$$
\AST(A)\times_\Q \Q_\ell \simeq G_\ell^1
$$
for every prime $\ell$.  In this case $\ST(A)$ can be defined as a maximal compact subgroup of the $\C$-points of $\AST(A)\times_\Q\C$, which depends neither on the choice of a prime $\ell$ nor on the choice of an embedding $\iota$.

By construction
$\ST(A)$ comes equipped with a faithful self-dual representation
$$
\rho: \ST(A)\rightarrow \GL(V),
$$
where $V$ is a $\C$ vector space of dimension $2g$. We call $\rho$ the standard representation of $\ST(A)$ and use it to view $\ST(A)$ as a compact real Lie subgroup of $\USp(2g)$.

In this note we are interested in the following three virtual characters of $\ST(A)$:
$$
a_1=\Tr\big( V\big)\,, \qquad a_2=\Tr\big(\wedge^2 V\big)\,,\qquad s_2=a_1^2-2a_2\,.
$$
For a nonnegative integer $j$, define the $j$th moment of a virtual character $\varphi$ as the virtual multiplicity of the trivial representation in $\varphi^{j}$. In particular, we have
\begin{align}\label{equation: intmoments}
\M_2[a_1]&=\dim_{\C}\big(V^{\otimes 2}\big)^{\ST(A)},\\\notag
\M_1[a_2]&=\dim_{\C}\big(\!\wedge^2\! V\big)^{\ST(A)},\\[4pt]\notag
\M_1[s_2]&=\M_2[a_1]-2\M_1[a_2].
\end{align}
Let $\End(A)$ denote the ring of endomorphisms of $A$ (defined over~$k$).


\begin{proposition}\label{proposition: rankends}
We have
$$
\M_2[a_1]=\rk_{\Z}(\End(A))\,.
$$
\end{proposition}

\begin{proof}
By Faltings isogeny theorem~\cite{Fal83}, we have
$$
\rk_{\Z}(\End(A))  =  \dim_{\Q_\ell}(\End(A)\otimes \Q_\ell) = \dim_{\Q_\ell}(\End_{G_\ell}(V_\ell(A)))\,.
$$
Observing that homotheties centralize $V_\ell(A)\otimes V_\ell(A)^{\vee}$ and Weyl's unitarian trick allows us to pass from $G^1_\ell$ to the maximal compact subgroup $\ST(A)$, we obtain
$$
\dim_{\Q_{\ell}}\big((V_\ell(A) \otimes V_\ell(A)^\vee)^{G_\ell}\big)=\dim_{\Q_{\ell}}\big((V_\ell(A) \otimes V_\ell(A)^\vee)^{G_\ell^1}\big)  =  \dim_{\C}\big((V \otimes V^\vee)^{\ST(A)}\big)\,.
$$
The proposition follows from the definition of $\M_2[a_1]$ and the self-duality of $V$.
\end{proof}

Let $\NS(A)$ denote the N\'eron--Severi group of $A$.

\begin{proposition}\label{proposition: rkPic}
We have
$$
\M_1[a_2]=\rk_\Z(\NS(A))\,.
$$
\end{proposition}

\begin{proof}
As explained in~\cite[\S2]{Tat65} (and in~\cite[Eq.\ (9)]{Tat66} using the same argument over finite fields), Faltings isogeny theorem provides an isomorphism
$$
\NS(A)\otimes_\Z \Q_\ell \simeq \big(H^2_{\et}(A_\Qbar,\Q_\ell)(1)\big)^{G_k}\simeq \big(\big(\wedge^2V_\ell(A)\big)(-1)\big)^{G_\ell}\,,
$$
where we have denoted Tate twists in the usual way and we have used the isomorphism $V_\ell(A)\simeq H^1_{\et}(A_\Qbar,\Q_\ell)(1)$. Then, as in the proof of Proposition~\ref{proposition: rankends}, we have
$$
\rk_{\Z}(\NS(A))  = \dim_{\Q_\ell}\big(\big(\wedge^2V_\ell(A)\big)(-1)\big)^{G_\ell^1}
 = \dim_{\C}\big(\wedge^2 V \big)^{\ST(A)}
 = \M_1[a_2],
$$
which completes the proof.
\end{proof}

In order to obtain a description of $\M[s_2]$, we will first relate $\rk_{\Z}(\End(A))$ with $\rk_{\Z}(\NS(A))$. There are three division algebras over $\R$: the quaternions $\HH$, the complex field $\C$, and the real field $\R$ itself. By Wedderburn's theorem we have
\begin{equation}\label{equation: real algebra}
\End(A)\otimes \R \simeq \prod_i \mat_{t_i}(\R)\times \prod_i \mat_{n_i}(\HH)\times \prod_i \mat_{p_i}(\C)\,,
\end{equation}
for some nonnegative integers $t_i$, $n_i$, $p_i$, where $\mat_n$ denotes the $n\times n$ matrix ring.

\begin{lemma}\label{lemma: rksrel}
With the notation of equation~\eqref{equation: real algebra}, we have
$$
\rk_\Z(\End(A))- 2\cdot \rk_\Z(\NS(A))=2\sum_i n_i - \sum_i t_i\,.
$$
In particular, we have the following inequality
\begin{equation}
2\cdot \rk_\Z(\NS(A))-g \leq \rk_\Z(\End(A))\leq 2\cdot \rk_\Z(\NS(A))+g\,.
\end{equation}
\end{lemma}

\begin{proof}
Let $\dagger$ denote the Rosati involution of $\End(A)\otimes\R$.  As explained in \cite[p.~190]{Mum70}, we have $\rk_\Z(\NS(A))=\dim_\R((\End(A)\otimes\R)^\dagger)$. For the first part of the lemma, it thus suffices to prove
\begin{equation}\label{equation: E2P}
  \dim_\R(\End(A)\otimes \R)-2\cdot  \dim_{\R}\big((\End(A)\otimes \R)^ \dagger\big)= 2 \sum_i n_i - \sum_i t_i\,.
\end{equation}
We say that an abelian variety defined over $k$ is isotypic if it is isogenous (over $k$) to the power of a simple abelian variety. Since both the left-hand and right-hand sides of \eqref{equation: E2P} are additive in the isotypic components of $A$, we may reduce to the case that $A$ is isotypic. We thus may assume that $A$ is the $r$th power of a simple abelian variety $B$. By Albert's classification of division algebras with a positive involution \cite[Thm. 2, \S21]{Mum70}, there are four possibilities for $\End(A)\otimes_{\Z} \R$, namely
$$
\I\ \mat_r(\R^e)\,,\qquad \II\ \mat_r(\mat_2(\R)^e)\,,\qquad \III\ \mat_r(\HH^e)\,,\qquad \IV\ \mat_r(\mat_d(\C)^ e)\,,
$$
where $e$ and $d$ are nonnegative integers. The action of the Rosati involution $\dagger$ on $\End(A)\otimes_{\Z} \R$ is also described in \cite[Thm.~2, \S21]{Mum70}, and the dimension of its fixed subspace can be easily read from the parameter $\eta$ listed on \cite[Table on p.~202]{Mum70}. The first part of the lemma then follows from the computations listed in Table \ref{table: dims}.

For the second part of the lemma we need to show that
$$
\left|2\sum_i n_i-\sum_i t_i\right|\leq g.
$$
This is immediate from Table \ref{table: dims} once we take into account that 
$e \leq \dim(B)$ for type $\I$, and $2e \leq \dim(B)$ for types $\II$ and $\III$ (see \cite[Table on p.~202]{Mum70}).
\end{proof}

\begin{table}[htb]
\caption{$\R$-algebra dimensions for isotypic $A$ by Albert type.}\label{table: dims}
\begin{tabular}{crrr}
$\mathrm{Type}$ & $\dim_\R(\End(A)\otimes \R)$ & $\dim_{\R}\big((\End(A)\otimes \R)^ \dagger\big)$ &  $2\sum_i n_i-\sum_i t_i$\\\toprule
$\I$ & $e r^2$ & $er(r+1)/2$ & $-er$\\[4pt]
$\II$ & $4er^ 2$ & $e(r+2r^ 2)$ & $-2er$ \\[4pt]
$\III$ & $4er^2$ & $e(-r+2r^ 2)$ & $2er$\\[4pt]
$\IV$ & $2er^ 2d^2$ & $er^ 2d^ 2$ & $0$ \\\bottomrule
\end{tabular}
\end{table}

As an immediate consequence of Proposition \ref{proposition: rankends}, Proposition \ref{proposition: rkPic}, and Lemma \ref{lemma: rksrel}, we obtain the following corollary.
\begin{corollary}
With the notation of equation~\eqref{equation: real algebra}, we have
\begin{equation*}
  \M_1[s_2] = 2 \sum_i n_i - \sum_i t_i\,.
\end{equation*}
\end{corollary}

\begin{remark}
The moment $\M_1[s_2]$ can also be interpreted as a Frobenius--Schur indicator, which allows us to give an alternative proof of \eqref{equation: E2P}, conditional on the Mumford--Tate conjecture, that does not make use of Albert's classification.
Recall that $\rho:\ST(A)\rightarrow \GL(V)$ denotes the standard representation of $\ST(A)$ and let $\Psi^2(\rho)$ be the central function defined as $\Psi^2(\rho)(g)=\rho(g^2)$ for every $g\in \ST(A)$; note that $s_2$ is simply $\Tr \Psi^2(\rho)$. Thus the moment $\M_1[s_2]$ is the Frobenius--Schur indicator $\mu(\rho)$ of the standard representation $\rho$, which is just the multiplicity of the trivial representation in $\Psi^2(\rho)$.
Inequality \eqref{equation: E2P} simply asserts that the trivial bound $|\mu(\rho)|\leq 2g$ can be improved to the sharper bound $|\mu(\rho)|\leq g$. Recall that the Frobenius--Schur indicator of an irreducible representation can only take the values $1$, $-1$, and $0$ depending on whether the representation is realizable over~$\R$, has real trace but it is not realizable over~$\R$, or has trace taking some value in $\C\setminus \R$, respectively (see \cite[p.~108]{Ser77}). To obtain the sharper bound, it suffices to show that any irreducible constituent $\sigma$ of the standard representation $\rho$ having real trace must have dimension at least $2$. This follows from our assumption that the Mumford--Tate conjecture holds for $A$.
\end{remark}

The results in this note explain, in particular, certain redundancies in Table~8 of \cite{FKRS12} that Seoyoung Kim used to prove Proposition~\ref{proposition: rankends} in the case where~$A$ is an abelian surface \cite[Proof of Thm.~3.4]{K}.

\section*{Acknowledgments.} The main results of this paper were discovered during the \emph{Arithmetic of Curves} workshop held at Baskerville Hall in Hay-on-Wye Wales in August 2018.  We thank the organizers Alexander Betts, Tim and Vladimir Dokchitser, and Celine Maistret for their kind invitation to participate. We also thank Seoyoung Kim for her interest in this note. The authors were financially supported by the Simons Collaboration in Arithmetic Geometry, Number Theory, and Computation via Simons Foundation grant 550033.

\end{document}